\documentclass[11pt]{amsart}

 \usepackage{amsfonts,graphics,amsmath,amsthm,amsfonts,amscd,amssymb,amsmath,latexsym,multicol,
 mathrsfs}
\usepackage{epsfig,url}
\usepackage{flafter}
\usepackage{fancyhdr}
\usepackage{hyperref}
\hypersetup{colorlinks=true, linkcolor=black}
\makeatletter

\def\jobis#1{FF\fi
  \def\predicate{#1}%
  \edef\predicate{\expandafter\strip@prefix\meaning\predicate}%
  \edef\job{\jobname}%
  \ifx\job\predicate
}

\makeatother

\if\jobis{proposal}%
\else
\fi

 \usepackage[matrix, arrow]{xy}

\DeclareMathOperator{\vol}{vol}

 \numberwithin{equation}{subsection}
 \numberwithin{footnote}{subsection}

 \newtheorem{lem}[subsection]{Lemma}
 \newtheorem{prop}[subsection]{Proposition}
 \newtheorem{thm}[subsection]{Theorem}

{%_%_%_%_% upright style; roman (non-italic) text
    \newtheoremstyle{upright}%
        {8pt plus2pt minus4pt}%
        {8pt plus2pt minus4pt}%
        {\upshape}%
        {}%
        {\bfseries\scshape}%
        {}%
        {1em}%
        {}%
\theoremstyle{upright}

}

 \newcommand{\PP}{\mathbb P}
 
 \newcommand{\Q}{\mathbb Q}
 \newcommand{\R}{\mathbb R}

 \newcommand{\rddown}[1]{\left\lfloor{#1}\right\rfloor} % round-down

\title{\large A\MakeLowercase{nticanonical volumes of} F\MakeLowercase{ano 4-folds}}
\author{\large C\MakeLowercase{aucher} B\MakeLowercase{irkar}}
\date{\today}
\begin{document}
\maketitle
\begin{abstract}
We find an explicit upper bound for the anticanonical volumes of Fano 4-folds with canonical singularities.
\end{abstract}

\section{Introduction}

We work over an algebraically closed field of characteristic zero.

Fano varieties constitute a fundamental class of algebraic varieties in algebraic geometry and many other fields. Due to their special features, it is more likely to have a detailed classification of Fano varieties compared to other classes such as Calabi-Yau varieties or varieties of general type. 

An important step in classifying Fano varieties is to obtain an explicit upper bound for their anticanonical volume under mild conditions on the singularities. Bounding the anticanonical volume of smooth Fano 3-folds goes back to Fano himself and it is a crucial step in showing that smooth Fano 3-folds form a bounded family. Similarly boundedness of anticanonical volume of smooth Fano varieties of fixed dimension is used to show that such Fano varieties are bounded: see Nadel [\ref{Nadel}] for the Picard number one case and Koll\'ar-Miyaoka-Mori [\ref{KMM}] for the general case which also uses Mori's bend and break technique. 

Not surprisingly everything gets more complicated when we allow singularities. The surface case is well-understood. An explicit upper bound for anticanonical volume of Fano 3-folds with canonical singularities is a quite recent result of Jiang-Zou [\ref{JZ}]: the upper bound is 324. In the $\Q$-factorial Picard number one case, an explicit upper bound was earlier found by Lai [\ref{Lai}] for $\epsilon$-log canonical ($\epsilon$-lc for short) Fano 3-folds. For more partial results in dimension 3, see [\ref{Kawamata}], [\ref{KMMT}], and the references in [\ref{JZ}]. 

In this note we find an explicit upper bound for Fano 4-folds with canonical singularities. 

\begin{thm}\label{t-main-4d}
Any Fano variety $X$ of dimension 4 with canonical singularities has 
$$
\vol(-K_X)=(-K_X)^4\le (104\mu(3,1)+8)^4
$$
where 
$$
\mu(3,1)=(840)^2(\frac{6(\mu(2,\frac{1}{2})+\frac{1}{2})}{\frac{1}{2}})^3
$$ 
is given by \ref{l-bnd-coeff-3d-canonical} below and in turn $\mu(2,\frac{1}{2})$ is given by the formula
$$
\mu(2,\delta)=(\frac{48}{\delta^2})2^{\frac{64}{\delta^3}}
$$
appearing in \ref{l-bnd-coeff-2d} below.
\end{thm}

To the best of our knowledge this is the first result of its kind in dimension 4 for singular Fano varieties. The proof closely follows the proof of [\ref{B-compl}, Theorem 1.6]. 
We did not aim to find an optimal bound but rather just an explicit bound. The 
upper bound in the theorem is unlikely to be anywhere close to the optimal bound.   
The number $\mu(3,1)$ is an explicit (not necessarily optimal) upper bound on cofficients of divisors $0\le B_V\sim_\R -K_V$ for Fano 3-folds $V$ with canonical singularities. Here 3 stands for dimension and 1 stands for 1-log canonical which is the same as canonical. A similar notation is used below in dimension 2.

In dimension 3 we prove a more general result. 

\begin{thm}\label{t-main-3d}
Let $\epsilon$ be a positive real number.
Let $X$ be a Fano variety of dimension 3 with $\epsilon$-lc singularities. Then for any 
$0<\delta<\epsilon$ we have  
$$
\vol(-K_X)\le (\frac{6(\mu(2,\delta)+\epsilon-\delta)}{\epsilon-\delta})^3.
$$
In particular, taking $\delta=\frac{\epsilon}{2}$, we have 
$$
\vol(-K_X)\le v(3,\epsilon):=(\frac{6(\mu(2,\frac{\epsilon}{2})+\frac{\epsilon}{2})}{\frac{\epsilon}{2}})^3.
$$
\end{thm}

Here the choice $\delta=\frac{\epsilon}{2}$ is arbitrary. When $\delta$ tends to $\epsilon$, the right hand side of the first inequality tends to $+\infty$. Similarly when $\delta$ tends to $0$ again the right hand side tends to $+\infty$. So the right hand side takes minimum for some value $\delta\in (0,\epsilon)$.

Chen Jiang informed us that he also has a proof of this theorem using different arguments. 

Thanks to the referees for their helpful comments.

\section{Proof of results}

We will use standard terminology in birational geometry regarding pairs, singularities, etc. 
Recall that a pair $(X,B)$ has $\epsilon$-log canonical ($\epsilon$-lc) singularities if its log discrepancies are at least $\epsilon$. When $B=0$ we just say that $X$ has $\epsilon$-lc singularities.

\begin{prop}\label{p-main}
Let $d\ge 2$ be a natural number and $\delta<\epsilon$ be positive real numbers. 
Assume that $X$ is a Fano variety of dimension $d$ with $\epsilon$-lc singularities and with $\vol(-K_X)>(2d)^d$. Then for any real number $a>\frac{d}{\sqrt[d]{\vol(-K_X)}}$, there is a pair $(V,\Omega_V)$ where 
\begin{itemize}
    \item $0<\dim V<d$, 
    \item $V$ is a $\delta$-lc Fano variety,
    \item $K_V+\Omega_V\sim_\Q 0$, and 
    \item there is a component of $\Omega_V$ with coefficient more than 
$$
\frac{(1-2a)(\epsilon-\delta)}{2a}.
$$
\end{itemize}
\end{prop} 
\begin{proof}
\emph{Step 1.}
In this step we introduce some notation. 
 Let $\alpha=\frac{d}{\sqrt[d]{\vol(-K_X)}}$. Then 
 $$
 \vol(-\alpha K_X)=\alpha^d\vol(-K_X)=d^d.
 $$ 
It is enough to prove the proposition for a rational number $a>\alpha$ sufficiently close to $\alpha$. Then $\vol(-a K_X)>d^d$. Since 
 $\vol(-K_X)>(2d)^d$ but 
 $$
 \vol(-2\alpha K_X)=2^d\vol(-\alpha K_X)=(2d)^d,
 $$ 
 we have $2\alpha<1$, so we can assume that $2a<1$.\\
 
\emph{Step 2.}
In this step we create a family of divisors and a covering family of subvarieties on $X$.
Since $\vol(-aK_X)>d^d$, there exists $0\le B\sim_\Q -{a}K_{X}$ such that 
$(X,B)$ is not klt (cf. [\ref{HMX2}, Lemma 3.2.2]). 
Pick a closed point $x\in X$ outside the non-klt locus of $(X,B)$. 
Then again since $\vol(-aK_X)>d^d$, there exists $0\le C\sim_\Q -aK_{X}$ such that 
$(X,C)$ is not klt at $x$. Changing $C$ up to $\Q$-linear equivalence we can assume that $(X,C)$ 
is not klt but lc at $x$. Perhaps increasing $a$ slightly and changing $C$ again 
we can assume that $(X,B+C)$
has a unique non-klt place whose centre, say $G$, contains $x$ [\ref{B-moduli-cy}, Lemma 2.16] (also see [\ref{HMX2}, Lemma 3.2.3] and its proof). Now put $\Delta=B+C$.

In the above construction $B$ is fixed but $\Delta$ depends on $x$. 
We have thus created a family of divisors $\Delta$ and a covering family of non-klt centres $G$ on $X$.  

By construction, 
$$
-(K_{X}+\Delta)=-(K_X+B+C)\sim_\Q -(K_X-aK_X-aK_X)=-(1-2a)K_X
$$ 
is ample as $2a<1$.  Therefore, the non-klt locus of $(X,\Delta)$ is connected, by the connectedness principle [\ref{FA}, Theorem 17.4]. This locus contains $G$ together with the non-klt locus of $(X,B)$. Since $x$ was chosen outside the non-klt locus of $(X,B)$ and since $G$ contains $x$, $G$ is not contained in the non-klt locus of $(X,B)$. Then $G$ intersects another non-klt centre of $(X,\Delta)$. This in particular means $\dim G>0$ because no other non-klt centre contains $x$.\\ 

\emph{Step 3.}
In this step we apply adjunction. 
From now on we assume that $G$ is a general member of the above covering family.
Let $F$ be the normalisation of $G$.
 By [\ref{HMX2}, Theorem 4.2] (also see [\ref{B-compl}, Construction 3.9 and Theorem 3.10]), we can write an adjunction formula  
$$
(K_{X}+\Delta)|_{F}\sim_\Q K_{F}+\Delta_{F}:=K_{F}+\Theta_{F}+P_{F}
$$
where $\Theta_{F}\ge 0$ and $P_{F}$ is pseudo-effective. Increasing $a$ slightly and adding to $\Delta$ we can assume $P_{F}$ is big and effective.

By assumption $\delta\in (0,\epsilon)$. 
Recall that $G$ intersects another non-klt centre of $(X,\Delta)$. Then we can choose $P_F$ such that we can assume $(F,\Delta_F)$ is not $\delta$-lc by [\ref{B-compl}, Lemma 3.14(2)].\\ 

\emph{Step 4.} 
In this step we define a boundary $\Pi_{F'}$ and a divisor $N_{F'}$.
Let $F'\to F$ be a log resolution of $(F,\Delta_{F})$. 
Let $K_{F'}+\Delta_{F'}$ be the pullback of $K_{F}+\Delta_{F}$. 
Define $\Pi_{F'}$ on $F'$ as follows. 
For each prime divisor $D$ on $F'$ define the coefficient 
$$ 
\mu_D\Pi_{F'}:= \left\{
  \begin{array}{l l}
    0 & \quad \text{if $\mu_D\Delta_{F'}< 0$,}\\
    \mu_D\Delta_{F'} & \quad \text{if $0\le \mu_D\Delta_{F'}\le 1-\delta$,}\\
    1-\delta & \quad \text{if $\mu_D\Delta_{F'}> 1-\delta$}
  \end{array} \right.
$$   
Clearly $(F',\Pi_{F'})$ is a klt pair, in fact, it is $\delta$-lc. 

Put
$$
N_{{F'}}:=\Delta_{F'}-\Pi_{F'}.
$$ 
Note that any component $D$ of $N_{F'}$ with negative coefficient is also a component of $\Delta_{F'}$ with negative coefficient. 
Since the components of $\Delta_{F'}$ with negative coefficient are exceptional over $F$, 
we deduce that the pushdown of $N_{F'}$ to $F$ is effective.\\

\emph{Step 5.}
In this step we consider a birational model $F''$ from which we obtain a Mori fibre space 
${F}'''\to T$.
Let $(F'',\Pi_{{F''}})$ be a log minimal model of $(F',\Pi_{F'})$ over $F$. 
We use $N_{F''},\Delta_{F''}$ to denote the pushdowns of $N_{F'},\Delta_{F'}$. We will use similar notation for other divisors and for pushdown to $F'''$ defined below.
By construction, 
$$
K_{F''}+\Pi_{{F''}}+N_{F''}=K_{F''}+\Delta_{F''}\sim_\Q 0/F,
$$
so $N_{F''}$ is anti-nef over $F$.
On the other hand, the pushdown of $N_{F''}$ to $F$ is effective.
So by the negativity lemma, $N_{F''}\ge 0$. In particular, $\Delta_{F''}\ge 0$.
Moreover, since $(F,\Delta_{F})$ is not $\delta$-lc, $(F'',\Delta_{F''})$ is not $\delta$-lc while $(F'',\Pi_{F''})$ is $\delta$-lc. Therefore, 
$N_{F''}\neq 0$.

Since $-(K_{X}+\Delta)$ is ample, $-(K_{F}+\Delta_{F})$ is ample, hence $-(K_{F''}+\Delta_{F''})$ 
is semi-ample and big. Pick a general 
$$
0\le L_{F''}\sim_\Q -(K_{F''}+\Delta_{F''})
$$  
so that $({F''},\Pi_{{F''}}+L_{F''})$ is $\delta$-lc. 
Now running an MMP on $K_{F''}+\Pi_{F''}+L_{F''}$ ends with a Mori fibre space ${F}'''\to T$ because 
$$
K_{F''}+\Pi_{{F''}}+L_{F''}+N_{F''}=K_{F''}+\Delta_{F''}+L_{F''}\sim_\Q 0
$$ 
and $N_{F''}\neq 0$.\\

\emph{Step 6.}
In this step we finish the proof.
By [\ref{HMX2}, Theorem 4.2][\ref{B-compl}, Theorem 3.12], we can write $K_X|_F=K_F+\Lambda_F$ where $(F,\Lambda_F)$ is sub-$\epsilon$-lc and $\Lambda_F\le \Delta_F$ ($\Lambda_F$ may have negative coefficients). Let $K_{F''}+\Lambda_{F''}$ be the pullback of $K_F+\Lambda_F$. Then $(F'',\Lambda_{F''})$ is also sub-$\epsilon$-lc, so the coefficients of $\Lambda_{F''}$ are $\le 1-\epsilon$. Moreover, $\Lambda_{F''}\le \Delta_{F''}$.

By construction $N_{F'''}$ is ample over $T$.
Let $D''$ be a component of $N_{F''}$ so that $D'''$ is ample over $T$. By the definition of $\Pi_{F'}$ and the fact $\Delta_{F''}\ge 0$, the components of $N_{F''}$ are exactly the components of $\Delta_{F''}$ with coefficient $>1-\delta$. So we have 
$$
\mu_{D''}(\Delta_{F''}-\Lambda_{F''})> 1-\delta-(1-\epsilon)=\epsilon-\delta.
$$ 

Note that 
$$
\Delta_{F''}-\Lambda_{F''}=(K_{F''}+\Delta_{F''})-(K_{F''}+\Lambda_{F''})
$$
$$
\sim_\Q
(K_X+\Delta)|_{F''}-(K_X|_{F''})
=\Delta|_{F''}\sim_\Q -2aK_{X}|_{F''}.
$$
On the other hand, 
$$
L_{F''}\sim_\Q -(K_X+\Delta)|_{F''}\sim_\Q -(1-2a)K_X|_{F''}
$$
$$
= \frac{1-2a}{2a}(-2aK_X|_{F''})
\sim_\Q \frac{1-2a}{2a}(\Delta_{F''}-\Lambda_{F''}).
$$

Now let $V$ be a general fibre of $F'''\to T$. Then from  
$$
K_{F'''}+\Omega_{{F'''}}
:=K_{F'''}+\Pi_{{F'''}}+\frac{1-2a}{2a}(\Delta_{F'''}-\Lambda_{F'''})+N_{F'''}
$$
$$
\sim_\Q K_{F'''}+\Pi_{{F'''}}+L_{F'''}+N_{F'''}=K_{F'''}+\Delta_{F'''}+L_{F'''}\sim_\Q 0
$$ 
and its restriction to $V$ we get $(V,\Omega_V)$ such that $K_V+\Omega_V\sim_\Q 0$. 
On the other hand, since $({F''},\Pi_{{F''}}+L_{F''})$ is $\delta$-lc, 
$({F'''},\Pi_{{F'''}}+L_{F'''})$ is $\delta$-lc, hence 
we see that $F'''$ is $\delta$-lc, so $V$ is a $\delta$-lc Fano variety with 
$0<\dim V<d$. 

Since $D'''$ is a component of $\Delta_{F'''}-\Lambda_{F'''}$ with coefficient 
$>\epsilon-\delta$ and since $D'''$ intersects $V$, we deduce that $\Omega_V$ has a 
component with coefficient more than 
$$
\frac{(1-2a)(\epsilon-\delta)}{2a}.
$$
 
\end{proof}

\begin{lem}\label{l-bnd-coeff-2d}
Let $\epsilon$ be a positive real number. Let $X$ be an $\epsilon$-lc Fano surface and let $B\ge 0$ be an $\R$-divisor with $K_X+B\sim_\R 0$. Then the coefficient of each component of $B$ is $\le \mu(2,\epsilon)$ where 
$$
\mu(2,\epsilon):=(\frac{48}{\epsilon^2})2^{\frac{64}{\epsilon^3}}.
$$
\end{lem}
\begin{proof}
By the proof of [\ref{Jiang-vol}, Theorem 2.8], any coefficient of $B$ is at most 
$$
l(\epsilon):=\frac{(2+4\epsilon)(4F_{\rddown{{{64}/{\epsilon^3}}}+2}-4)}{\epsilon^2}
$$
where $F_n$ denotes the Fibonacci number defined by $F_0=F_1=1$ and $F_n=F_{n-1}+F_{n-2}$ for $n\ge 2$. 
Inductively we can easily see that $F_n\le 2^{n-1}$. So  
$$
l(\epsilon)\le \frac{24F_{\rddown{{{64}/{\epsilon^3}}}+2}}{\epsilon^2}  
\le \mu(2,\epsilon):=(\frac{48}{\epsilon^2})2^{{64}/{\epsilon^3}}.
$$
\end{proof} 
 
\begin{proof}(of Theorem \ref{t-main-3d})
Let $X$ be a Fano 3-fold with $\epsilon$-lc singularities. The right hand side of the both inequalities in \ref{t-main-3d} are more than $6^3$, so it is enough to treat the case when $\vol(-K_X)>6^3$. 
Pick a positive real number $\delta<\epsilon$ and  
pick a real number $a>\frac{3}{\sqrt[3]{\vol(-K_X)}}$. 
Applying Proposition \ref{p-main}, there is a pair $(V,\Omega_V)$ where 
\begin{itemize}
    \item $0<\dim V<3$, 
    \item $V$ is a $\delta$-lc Fano variety,
    \item $K_V+\Omega_V\sim_\Q 0$, and 
    \item there is a component of $\Omega_V$ with coefficient more than 
$$
\frac{(1-2a)(\epsilon-\delta)}{2a}.
$$
\end{itemize}
So $\dim V=1$ or $2$. If $\dim V=1$, then $V\simeq \PP^1$, so  
$$
\frac{(1-2a)(\epsilon-\delta)}{2a}<\mu(1,\delta):=2. 
$$
On the other hand, if $\dim V=2$, then by Lemma \ref{l-bnd-coeff-2d}, 
we have 
$$
\frac{(1-2a)(\epsilon-\delta)}{2a}< \mu(2,\delta)=(\frac{48}{\delta^2})2^{\frac{64}{\delta^3}}.
$$
Note that $\mu(1,\delta)<\mu(2,\delta)$. 
Thus we can calculate that 
$$
\frac{1}{a}<\frac{2(\mu(2,\delta)+\epsilon-\delta)}{\epsilon-\delta}.
$$
Fixing $\delta$ and taking the limit when $a$ approaches $\alpha:=\frac{3}{\sqrt[3]{\vol(-K_X)}}$, we see that 
$$
\frac{1}{\alpha}\le \frac{2(\mu(2,\delta)+\epsilon-\delta)}{\epsilon-\delta}.
$$
This in turn gives 
$$
\vol(-K_X)\le (\frac{2(\mu(2,\delta)+\epsilon-\delta)}{\epsilon-\delta})^33^3.
$$
Applying this to $\delta:=\frac{\epsilon}{2}$, we have 
$$
\vol(-K_X)\le (\frac{2(\mu(2,\frac{\epsilon}{2})+\frac{\epsilon}{2})}{\frac{\epsilon}{2}})^33^3.
$$
\end{proof}

The next lemma is preparation for the proof of boundedness of volume in dimension 4. 
 
\begin{lem}\label{l-bnd-coeff-3d-canonical}
Let $X$ be a Fano 3-fold with canonical singularities and let $B\ge 0$ be an $\R$-divisor with $K_X+B\sim_\R 0$. Then the coefficient of each component of $B$ is $\le \mu(3,1)$ where 
$$
\mu(3,1):=(840)^2v(3,1)
=(840)^2(\frac{6(\mu(2,\frac{1}{2})+\frac{1}{2})}{\frac{1}{2}})^3.
$$ 
\end{lem} 
\begin{proof}
 By applying [\ref{Chen-Jiang}, Proposition 2.4] to a terminal crepant model of $X$, we deduce that $IK_X$ is Cartier for some natural number $I\le 840$. 
 On the other hand, we need an upper bound for the $\vol(-K_X)$. Such a bound is given by Theorem \ref{t-main-3d} which is 
 $$
 v(3,1)=(\frac{6(\mu(2,\frac{1}{2})+\frac{1}{2})}{\frac{1}{2}})^3.
 $$
One could also use the upper bound $\vol(-K_X)\le 324$ by [\ref{JZ}] but to make the theorem logically independent of [\ref{JZ}] we will use $v(3,1)$. 

Let $D$ be a component $B$. Then 
$$
\mu_DB \le (\mu_DB) D\cdot (-IK_{X})^2\le B\cdot (-IK_{X})^2
$$
$$
=(I)^2(-K_{X})^3\le (840)^2v(3,1). 
$$

\end{proof}
 
\begin{proof}(of Theorem \ref{t-main-4d})
Let $X$ be a Fano 4-fold with canonical singularities. We can assume that $\vol(-K_X)>8^4$. Pick a positive real number $\delta\in (\frac{12}{13},1)$ and  
pick a real number $a>\alpha:=\frac{4}{\sqrt[4]{\vol(-K_X)}}$. 
Applying Proposition \ref{p-main}, there is a pair $(V,\Omega_V)$ where 
\begin{itemize}
    \item $0<\dim V<4$, 
    \item $V$ is a $\delta$-lc Fano variety,
    \item $K_V+\Omega_V\sim_\Q 0$, and 
    \item there is a component of $\Omega_V$ with coefficient more than 
$$
\frac{(1-2a)(1-\delta)}{2a}.
$$
\end{itemize}
Since $\delta\in (\frac{12}{13},1)$ and since $V$ is $\delta$-lc of dimension at most $3$, 
$V$ actually has canonical singularities, by [\ref{LX}][\ref{Jiang}]. 
Therefore, considering the cases $\dim V=1,2,3$ separately, we have 
$$
\frac{(1-2a)(1-\delta)}{2a}<\max\{\mu(1,1), \mu(2,1), \mu(3,1)\}=\mu(3,1).
$$
So we get 
$$
\frac{1}{a}<\frac{2(\mu(3,1)+1-\delta)}{1-\delta}.
$$ 
Taking limit as $a$ approaches $\alpha$ we then have 
$$
\frac{1}{\alpha}=\frac{\sqrt[4]{\vol(-K_X)}}{4}\le \frac{2(\mu(3,1)+1-\delta)}{1-\delta}.
$$
In turn taking limit when $\delta$ approaches $\frac{12}{13}$ we see that 
$$
\vol(-K_X)\le (\frac{8(\mu(3,1)+\frac{1}{13})}{\frac{1}{13}})^4=(104\mu(3,1)+8)^4.
$$
We can now apply Lemma \ref{l-bnd-coeff-3d-canonical} to get an explicit bound. 

\end{proof}

%%%%%%%%%%%%%%%%%%%%%%%%%%%%%%%%%%%%%

\bigskip
\bigskip

Yau Mathematical Sciences Center, Jing Zhai, Tsinghua University, Hai Dian District, Beijing, China 100084; 

Email: birkar@tsinghua.edu.cn


\begin{thebibliography}{99}

%\bibitem{}\label{AM} 
%V. Alexeev and S. Mori, \emph{Bounding singular surfaces of general type}, in
%Algebra, Arithmetic and Geometry with Applications (West Lafayette, IN,
%2000), Springer, Berlin, 2004, 143--174.

\bibitem{}\label{B-moduli-cy}  
C. Birkar, \emph{Geometry and moduli of polarised varieties}. arXiv:2006.11238v1.

\bibitem{}\label{B-compl}  
C. Birkar, \emph{Anti-pluricanonical systems on Fano varieties}, Ann. of Math. (2)
190 no. 2 (2019), 345--463.

\bibitem{}\label{Chen-Jiang}
M. Chen, C. Jiang, \emph{On the anti-canonical geometry of Q-Fano threefolds}, J. Differential Geom. 104 (2016), no.1, 59--109.

\bibitem{}\label{Jiang}
C. Jiang, \emph{A gap theorem for minimal log discrepancies of non-canonical singularities
in dimension three}, arXiv:1904.09642v1.

\bibitem{}\label{Jiang-vol}
C. Jiang, \emph{Boundedness of anti-canonical volumes of singular log Fano threefolds}, arXiv:1411.6728.

\bibitem{}\label{JZ}
C. Jiang, Y. Zou, \emph{An effective upper bound for anti-canonical volumes of canonical Q-Fano threefolds }, arXiv:2107.01037.

\bibitem{}\label{HMX2}
C.~D.~Hacon, J.~M$^{\rm c}$Kernan and C.~Xu,
\emph{ACC for log canonical thresholds}, Ann. of Math. (2) \textbf{180} (2014), no. 2, 523--571.

\bibitem{}\label{Kawamata}
Y. Kawamata, \emph{Boundedness of $\Q$-Fano threefolds}, in Proceedings of the International Conference on Algebra, Part 3 (Novosibirsk, 1989), Contemp. Math.
131, Amer. Math. Soc., Providence, RI, 1992, pp. 439--445.

\bibitem{}\label{FA}
J. Koll\a'ar, \emph{Flips and Abundance for Algebraic Threefolds}, Ast\'erisque 211, Soc.
Math. de France, Paris, 1992, (with 14 co-authors; papers from the Second Summer Seminar on Algebraic Geometry held at the University of Utah, Salt Lake
City, Utah, August 1991).

\bibitem{}\label{KMM}
J. Koll\'ar, Y. Miyaoka, and S. Mori, \emph{Rational connectedness and boundedness of Fano manifolds}, J. Differential Geom. 36 no. 3 (1992), 765--779.

\bibitem{}\label{KMMT}
J. Koll\'ar, Y. Miyaoka, and S. Mori, H. Takagi, \emph{Boundedness of canonical
$\Q$-Fano 3-folds}, Proc. Japan Acad. Ser. A Math. Sci. 76 no. 5 (2000), 73--77.


\bibitem{}\label{Lai}
C.-J. Lai, \emph{Bounding volumes of singular Fano threefolds}, Nagoya Math. J. 224
no. 1 (2016), 37--73.

\bibitem{}\label{LX}
J. Liu, L. Xiao, \emph{An optimal gap of minimal log discrepancies of threefold non-canonical singularities}, arXiv:1909.08759v2.

\bibitem{}\label{Nadel}
A. M. Nadel, \emph{The boundedness of degree of Fano varieties with Picard number
one}, J. Amer. Math. Soc. 4 no. 4 (1991), 681--692.

\end{thebibliography}
\end{document}